\documentclass[12pt]{amsart}

\usepackage{amssymb}

\newtheorem{theorem}{Theorem}[section]
\newtheorem{proposition}[theorem]{Proposition}
\newtheorem{lemma}[theorem]{Lemma}
\newtheorem{corollary}[theorem]{Corollary}

\theoremstyle{definition}
\newtheorem{definition}[theorem]{Definition}

\newtheorem{remark}[theorem]{Remark}

\newtheorem{example}[theorem]{Example}

\newcommand{\ZZ}{ \ensuremath{\mathbb{Z}}}

\newcommand{\Gin}{\ensuremath{\mathrm{Gin}}}

\newcommand{\reg}{\ensuremath{\mathrm{reg}}}
\newcommand{\Tor}{\ensuremath{\mathrm{Tor}}}

\newcommand{\Ext}{\ensuremath{\mathrm{Ext}}}

\newcommand{\mideal}{\ensuremath{\mathfrak{m}}}

\newcommand{\lex}{{\mathrm{lex}}}

\def\cocoa{{\hbox{\rm C\kern-.13em o\kern-.07em C\kern-.13em o\kern-.15em A}}}

\newcommand{\lk}{\mathrm{lk}}
\newcommand{\yy}{\mathbf{y}}
\newcommand{\pol}{\mathrm{pol}}

\begin{document}

\title[Simplicial complexes with Serre's conditions]
{h-vectors of simplicial complexes\\
 with Serre's conditions}

\author{Satoshi Murai}
\address{
Satoshi Murai,
Department of Mathematical Science,
Faculty of Science,
Yamaguchi University,
1677-1 Yoshida, Yamaguchi 753-8512, Japan.
}

\author{Naoki Terai}
\address{
Naoki Terai,
Department of Mathematics, Faculty of Culture and Education,
Saga University, Saga, 840-8502, Japan.
}



\begin{abstract}
We study $h$-vectors of simplicial complexes which satisfy Serre's condition ($S_r$).
Let $r$ be a positive integer.
We say that a simplicial complex $\Delta$ satisfies Serre's condition ($S_r$)
if $\tilde H_i(\lk_\Delta(F);K)=0$ for all $F \in \Delta$ and for all $i < \min \{r-1,\dim \lk_\Delta(F)\}$,
where $\lk_\Delta(F)$ is the link of $\Delta$ with respect to $F$
and where $\tilde H_i(\Delta;K)$ is the reduced homology groups of $\Delta$
over a field $K$.
The main result of this paper is that if $\Delta$ satisfies Serre's condition ($S_r$) then 
(i) $h_k(\Delta)$ is non-negative for $k =0,1,\dots,r$
and (ii) $\sum_{k\geq r}h_k(\Delta)$ is non-negative.
\end{abstract}

\maketitle

\section{Introduction}

The study of $h$-vectors of simplicial complexes is an interesting research area in combinatorics as well as in combinatorial commutative algebra.
One of fundamental questions on $h$-vectors is ``when $h$-vectors are non-negative?".
A classical result of Stanley guarantees that
$h$-vectors of Cohen--Macaulay complexes are non-negative.
In this paper, we generalize this classical result
in terms of Serre's conditions, which appear in commutative algebra.

A \textit{simplicial complex} $\Delta$ on $[n]=\{1,2,\dots,n\}$
is a collection of subsets of $[n]$
satisfying that (i) $\{i\} \in \Delta$ for all $i \in [n]$
and (ii) if $F \in \Delta$ and $G \subset F$ then $G \in \Delta$.
An element $F$ of $\Delta$ is called a \textit{face} of $\Delta$
and maximal faces (under inclusion) are called \textit{facets} of $\Delta$.
We say that $\Delta$ is \textit{pure} if all facets of $\Delta$ have the same cardinality.
Let $f_k (\Delta)$ be the number of faces $F \in \Delta$ with $|F|=k+1$,
where $|F|$ is the cardinality of $F$.
The \textit{dimension} of $\Delta$ is 
$\dim \Delta = \max \{ k : f_k(\Delta) \ne 0\}$.
The vector $f(\Delta) = (f_0(\Delta),f_1(\Delta), \dots, f_{d-1}(\Delta))$ is called the \textit{$f$-vector} of $\Delta$,
where $d= \dim \Delta +1$.
The {\em $h$-vector} $h(\Delta)=(h_0(\Delta),h_1(\Delta),\dots,h_d(\Delta))$ of
$\Delta$ is  defined by the relation
$\sum_{i=0}^{d} f_{i-1}(\Delta)(x-1)^{d-i}
=\sum_{i=0}^{d} h_{i}(\Delta)x^{d-i},$
where $f_{-1}(\Delta) =1$.
Thus knowing $f(\Delta)$ and knowing $h(\Delta)$ are equivalent.

We introduce Serre's condition ($S_r$) for simplicial complexes.
Let $\tilde H_i(\Delta;K)$ be the reduced homology groups of $\Delta$
over a field $K$. 
The \textit{link} of 
$\Delta$ with respect to a face $F \in \Delta$ is the simplicial complex
$\lk_\Delta(F)=\{G \subset [n]\setminus F: G \cup F \in \Delta\}$.
Let $r \geq 1$ be an integer.
We say that a $(d-1)$-dimensional simplicial complex $\Delta$ satisfies \textit{Serre's condition} ($S_r$) (over $K$)
if, for every face $F \in \Delta$ (including the empty face),
$\tilde H_i(\lk_\Delta(F);K)=0$ for $i < \min \{r-1,\dim \lk_\Delta(F)\}$.

We remark some basic facts.
Every simplicial complex satisfies ($S_1$).
On the other hand,
for $r \geq 2$,
simplicial complexes satisfying ($S_r$) are pure and strongly connected.
($S_2$) states that $\Delta$ is pure and $\lk_\Delta(F)$ is connected for all faces $F \in \Delta$
with $|F|<\dim \Delta$.
($S_d$) is equivalent to the famous Cohen--Macaulay property of simplicial complexes (see \cite{St2}).
Also, the above combinatorial definition of ($S_r$) is equivalent to
that of usual Serre's condition ($S_r$) in commutative algebra by the Stanley--Reisner correspondence.

A classical influential result of Stanley states that
if a $(d-1)$-dimensional simplicial complex $\Delta$ is Cohen--Macaulay
(that is, if it satisfies ($S_d$))
then $h_k(\Delta)$ are non-negative for all $k$.
We generalize this Stanley's result in the following way.

\begin{theorem}
\label{1}
If a simplicial complex $\Delta$ satisfies Serre's condition $(S_r)$ then $h_k(\Delta)$
is non-negative for $k=0,1,\dots,r$.
\end{theorem}

Actually, we prove that $(h_0,h_1,\dots,h_r)$ is an $M$-vector,
that is, the $f$-vector of a multicomplex (see \cite[p.\ 56]{St2}).
In particular, this fact shows that if $h_t=0$ for some $t \leq r$ then $h_k=0$ for all $t \leq k \leq r$.
We prove the following stronger statement on vanishing of $h$-vectors.

\begin{theorem}
\label{3-5}
Let $\Delta$ be a simplicial complex which satisfies Serre's condition $(S_r)$.
If $h_t(\Delta)=0$ for some $t \leq r$ then $h_k(\Delta)=0$ for all $k \geq t$ and
$\Delta$ is Cohen--Macaulay.
\end{theorem}

The non-negativity in Theorem \ref{1} is sharp.
For all integers $2 \leq r <d$,
there exists a ($d-1$)-dimensional simplicial complex $\Delta$ which satisfies Serre's condition ($S_r$)
but $h_{r+1}(\Delta)<0$ (\cite[Example 3.5]{TY}).
On the other hand,
we prove the following weak non-negative property,
which shows that negative parts of $h$-vectors are not very big.

\begin{theorem} 
\label{main3}
If a $(d-1)$-dimensional simplicial complex $\Delta$ satisfies Serre's condition $(S_r)$ then  $h_r(\Delta)+h_{r+1}(\Delta)+ \cdots + h_d(\Delta)$
is non-negative.
\end{theorem}

Our proofs of the above theorems
are based on commutative algebra theory.
In particular, to prove Theorem \ref{1},
we prove the following algebraic result which is itself of interest.
Let $S=K[x_1,\dots,x_n]$ be a polynomial ring over an infinite field $K$
with each $\deg x_i=1.$
For a graded $S$-module $N$ and an integer $j \in \ZZ$,
we denote by $N_j$ the graded component of degree $j$ of $N$.
Let $I \subset S$ be a homogeneous ideal and $R=S/I$.
The \textit{Hilbert series} of $R$ is the formal power series
$F(R,\lambda) = \sum_{q=0}^{\infty} (\dim_K R_q) \lambda^q$.
It is known that $F(R,\lambda)$ is a rational function
of the form 
$(h_0 + h_1 \lambda + \cdots + h_s \lambda^s) / (1 -
\lambda)^d$,
where each $h_i$ is an integer with $h_s \neq 0$
and where $d$ is the Krull dimension of $R$.
See \cite[Corollary 4.1.8]{BH}.
The vector $(h_0(R), h_1(R), \ldots, h_s(R)) = (h_0, h_1, \ldots, h_s)$ is called the
{\em $h$-vector} of $R$.
For convenience, we define $h_k(R)=0$ for $k > s$. 
Let $M$ be a finitely generated graded $S$-module.
The \textit{graded Betti numbers} $\beta_{i,j}(M)$ of $M$ are integers defined by
$\beta_{i,j}(M)=\dim_K \Tor_i(M,K)_j$.
The \textit{(Castelnuovo-Mumford) regularity} of $M$ is
the integer
$$\reg\hspace{1pt} M =\max\big\{j : \beta_{i,i+j}(M) \ne 0 \mbox{ for some }i\big\}.$$
We denote by $\omega_S$ the canonical module of $S$.

\begin{theorem}
\label{2}
Let $r \geq 1$ be an integer,
$I \subset S$ a homogeneous ideal and
$d$ the Krull dimension of $R=S/I$.
Suppose that $\reg(\Ext_S^{n-i}(R,\omega_S)) \leq i-r$ for $i=0,1,\dots,d-1$.
There exists
a linear system of parameters $\Theta =\theta_1,\dots,\theta_d$ of $R$ such that
$$h_k(R)=\dim_K (R/\Theta R)_k \ \ \mbox{ for $k \leq r$}.$$
\end{theorem}


This paper is organized as follows:
In Section 2, we prove Theorems \ref{1} and \ref{2} by purely algebraic methods.
In Section 3, by using Swartz's equations on $h$-vectors of pure simplicial complexes, we prove Theorems \ref{3-5} and \ref{main3}.
In addition, a few examples of simplicial complexes satisfying Serre's conditions are given in this section.
In Section 4, we study monomial ideals and graded Betti numbers.
First, we generalize Theorems \ref{1}, \ref{3-5} and \ref{main3}
to monomial ideals by using polarization trick.
Second, we give some upper bounds of graded Betti numbers
for monomial ideals in terms of Serre's conditions.

\section{Non-negativity of $h$-vectors}

We refer the reader to \cite{St2} for the basics on commutative algebra theory.
Throughout this paper, we assume that $K$ is an infinite field.

Let $M$ be a finitely generated graded $S$-module.
For a sequence $\yy =y_1,\dots,y_\ell \in S_1$ of linear forms,
we set
$$A_k(\yy;M)= \left(0:_{M/(y_1,\dots,y_k)M} y_{k+1} \right) =\big\{f \in M/(y_1,\dots,y_k)M: y_{k+1} f=0\big\}$$
for $k=0,1,\cdots,\ell-1$.
The following fact is well-known (cf.\ \cite[Section 3]{St1}).

\begin{lemma}
\label{3}
Let $I \subset S$ be a homogeneous ideal and
$d$ the Krull dimension of $R=S/I$.
Let $\Theta=\theta_1,\dots,\theta_d$ be a linear system of parameters of $R$.
If $\dim_K A_k(\Theta;R)_j=0$ for all $k=0,1,\dots,d-1$ and for all $j \leq r-1$ then
\begin{itemize}
\item[(i)] $h_j(R)= \dim_K(R/\Theta R)_j$
 for $j \leq r.$
\item[(ii)] $h_{r+1}(R)= \dim_K(R/\Theta R)_{r+1} -\sum_{k=0}^{d-1} \dim_K A_k(\Theta;R)_r.$
\end{itemize}
\end{lemma}

We say that a sequence $\yy=y_1,\dots,y_\ell$ of linear forms
is an \textit{almost regular sequence} on $M$ if $A_k(\yy;M)$ has finite length for $k=0,1,\dots,\ell-1$.
The set of almost regular sequences $\yy=y_1,\dots,y_\ell$ on $M$ contains a non-empty Zariski open subset
of $S_1 \bigoplus_K \cdots \bigoplus_K S_1 \cong K^{\ell \times n}$
(see \cite{AH}).
Also, if $y_1,\dots,y_d$ is an almost regular sequence on $M$, where $d$
is the Krull dimension of $M$,
then $y_1,\dots,y_d$ forms a system of parameters of $M$.

Let $H_i(\yy;M)$ (respectively, $H^i(\yy;M)$) be the $i$-th Koszul homology (respectively, Koszul cohomology) of $M$
with respect to a sequence $\yy=y_1,\dots,y_\ell$.
A key lemma to prove Theorem \ref{2} is the next result due to
Aramova and Herzog \cite[Theorem 1.1]{AH}.

\begin{lemma}[Aramova-Herzog]
\label{4}
Let $M$ be a finitely generated graded $S$-module of Krull dimension $d$
and $\yy=y_1,\dots,y_n$ an almost regular sequence on $M$ which forms a $K$-basis of $S_1$.
Then
$H_i(y_1,\dots,y_k;M)$ has finite length
and $H_i(y_1,\dots,y_k;M)_{i+j} = 0$ for $j > \reg\hspace{1pt} M$
in the following cases:
\begin{itemize}
\item[(i)] $i \geq 1$ and $k=1,2,\dots,n$;
\item[(ii)] $i=0$ and $k=d,d+1,\dots,n$.
\end{itemize}
\end{lemma}
\pagebreak

\begin{proof}
The finite length property is easy (cf.\ \cite[Lemma 3.2]{Sc}).
We explain how to apply \cite[Theorem 1.1]{AH} for vanishing of Koszul homologies.
For $\ell=1,2,\dots,n$,
let
$$s_\ell = \max\big\{j:A_{\ell-1}(\yy;M)_j \ne 0\big\}$$
and
$$r_\ell = \max \big\{j: H_i(y_1,\dots,y_\ell;M)_{i+j} \ne 0 \mbox{ for some }i \geq 1\big\}.$$
Aramova and Herzog \cite[Theorem 1.1]{AH} proved $r_\ell=\max\{s_1,\dots,s_\ell\}$ for all $\ell$.

Since $\Tor_i(M,K)_j\cong H_i(\yy;M)_j$, $r_n \leq \reg\hspace{1pt} M$.
Then $r_1 \leq \cdots \leq r_n \leq \reg\hspace{1pt} M$.
This proves (i).
Also, for $k=d,d+1,\dots,n-1$, since $H_0(y_1,\dots,y_k;M)=M/(y_1,\dots,y_k)M$ has finite length,
$$
\max\big\{j: H_0(y_1,\dots,y_k;M)_j \ne 0\big\}
= \max\big\{j: A_k(\yy;M)_j \ne 0\big\}=s_{k+1} \leq r_n \leq \reg\hspace{1pt} M.$$
Finally, $H_0(\yy;M)_j \cong \Tor_0(M,K)_j =0$ for $j > \reg\hspace{1pt}  M$ is obvious.
\end{proof}

Let $H_\mideal^i(M)$ be the $i$-th local cohomology module of $M$.
Another key lemma is the next result due to Schenzel \cite{Sc,S2}.

\begin{lemma}[Schenzel]
\label{5}
Let $M$ be a finitely generated graded $S$-module of Krull dimension $d$
and $\yy =y_1,\dots,y_p$ an almost regular sequence on $M$,
where $p \geq 1$ is an integer.
Then, for all $i \geq 1$ and $j \in \ZZ$, one has
$$\dim_K H_i(\yy;M)_j \leq \sum_{\ell=0}^{\min\{p-i,d\}} \dim_K H_{i+\ell}\big(\yy;H_\mideal^\ell(M)\big)_j.$$
\end{lemma}

\begin{proof}
We sketch the proof since the result was essentially proved in \cite{Sc}.
Let $C^{\bullet}$ be the \v{C}ech complex and $K_\bullet (\yy;M)$ the Koszul complex of $M$ with respect to $\yy$.
Define the double complex $D^{\bullet \bullet}$ such that $D^{s,t}= C^s \bigotimes_S K_{p-t}(\yy;M)$.
Note that all maps in $D^{\bullet \bullet}$ are homogeneous of degree $0$.
There are two spectral sequences
(we follow the notation of \cite[Section A3]{Ei})
\begin{eqnarray*}
_{\mathrm{vert}}^{\hspace{10pt} 2}E^{s,t}=H_\mideal^s \big(H_{p-t}(\yy;M)\big) &\Rightarrow& H^{s+t}\big(\mathrm{tot}(D^{\bullet \bullet})\big)\\
^{\hspace{8pt} 2}_{\mathrm{hor}}E^{s,t}=H_{p-t}\big(\yy;H_\mideal^s(M) \big) &\Rightarrow& H^{s+t}\big(\mathrm{tot}(D^{\bullet \bullet})\big).
\end{eqnarray*}
By Lemma \ref{4}(i), $H_{p-t}(\yy;M)$ has finite length if $t \ne p$.
Thus, by the basic properties of local cohomology,
$$_{\mathrm{vert}}^{\hspace{10pt} 2}E^{s,t}=0
\ \ \mbox{ if } (s,t) \not\in \big\{(0,0),(0,1),\dots,(0,p),(1,p),(2,p),\dots,(d,p)\big\}$$
and
$$^{\hspace{10pt} 2}_{\mathrm{vert}}E^{0,t}=H_{p-t}(\yy;M)
\ \ \mbox{ for } t=0,1,\dots,p-1.$$
Then this spectral sequence degenerates at $^2E$ and 
$H^k(\mathrm{tot}(D^{\bullet \bullet})) \cong H_{p-k}(\yy;M)$
for $k=0,1,\dots,p-1$.
Since $\dim_K H^k(\mathrm{tot}(D^{\bullet \bullet}))_j \leq \dim_K(\bigoplus_{s+t=k} {}^{\hspace{8pt} 2}_{\mathrm {hor}}E^{s,t})_j$ for all $j \in \ZZ$,
we get the desired inequality.
\end{proof}

\begin{remark}
In \cite{Sc}, Schenzel assumed $p \geq d$ since he also considered an estimate for $H_0(\yy;M)$
and since $H_0(\yy;M) \ne {}^{\hspace{10pt} 2}_{\mathrm {vert}}E^{0,p}$ unless it has finite length.
However, we do not need this assumption since we are assuming $i \geq 1$.
\end{remark}
\pagebreak

The next statement and Lemma \ref{3}(i) prove Theorem \ref{2}.

\begin{proposition}
\label{prop}
With the same notation as in Theorem \ref{2},
there exists an almost regular sequence $\yy=y_1,\dots,y_n$ on $R$ such that
$A_k(\yy;R)_j=0$
for $k=0,1,\dots, d-1$ and $j\leq r-1$.
\end{proposition}

\begin{proof}
Let $\yy=y_1,\dots,y_n$ be an almost regular sequence on $R$ such that
\begin{itemize}
\item $\yy$ forms a $K$-basis of $S_1$;
\item $\yy$ is an almost regular sequence on $\Ext^{i}_S(R,\omega_S)$ for all $i$.
\end{itemize}
Observe that there is a natural surjection
\begin{eqnarray*}
\begin{array}{cccc}
H_1(y_1,\dots,y_{k+1};R)_j & \to & A_k(\yy;R)_{j-1}\smallskip\\
{[} f_1 e_1 + \cdots +f_{k+1} e_{k+1}] & \to & [ f_{k+1} ]
\end{array}
\end{eqnarray*}
for all $k=0,1,\dots,n-1$ and $j \in \ZZ$,
where each $f_\ell \in R$ and each $e_\ell$ is a basis element of $K_1(y_1,\dots,y_{k+1};R)$.
It is enough to prove $H_1(y_1,\dots,y_{k+1};R)_j=0$ for all $k=0,1,\dots,d-1$ and $j \leq r$.
Moreover, Lemma \ref{5} shows that it is enough to prove
$$H_{1+i}\big(y_1,\dots,y_{k+1};H_\mideal^i(R)\big)_j=0$$
for all $0 \leq i \leq k \leq d-1$ and for all $j \leq r$.

Fix integers $0 \leq i \leq k \leq d-1$ and $j \leq r$.
We denote by $N^\vee$ the Matlis dual of a graded $S$-module $N$. 
By the local duality and the self duality of the Koszul complex,
\begin{eqnarray*}
H_{1+i}\big(y_1,\dots,y_{k+1};H_\mideal^i(M)\big)_j 
&\cong& H_{1+i}\big(y_1,\dots,y_{k+1};H_\mideal^i(M) \big)^\vee_{-j}\\
&\cong& H^{1+i}\big(y_1,\dots,y_{k+1};\Ext^{n-i}_S(R,\omega_S)\big)_{-j}\\
&\cong& H_{k-i}\big(y_1,\dots,y_{k+1};\Ext^{n-i}_S(R,\omega_S)\big)_{(k-i)+i+1 -j}.
\end{eqnarray*}
(We use
$K_\bullet(y_1,\dots,y_{k+1};N)^\vee 
\cong \mathrm{Hom}_S(K_\bullet(y_1,\dots,y_{k+1};S),N^\vee)$
for the second equation.)
Since $i+1-j>i-r$,
Lemma \ref{4}(i) and the assumption $\reg( \Ext_S^{n-i}(R,\omega_S)) \leq i-r$ show that
the above vector space vanishes when $i <k$.
Also, since the Krull dimension of $\Ext_S^{n-i}(R,\omega_S)$ is smaller than or equal to $i$,
it also vanishes when $i =k$ by Lemma \ref{4}(ii).
\end{proof}

Finally, we prove Theorem \ref{1}.
The next fact is well-known.

\begin{lemma}
\label{pure}
If a simplicial complex $\Delta$ satisfies $(S_r)$ and $r \geq 2$ then $\Delta$ is pure.
\end{lemma}

\begin{proof}
We use induction on $\dim \Delta$.
Since $\tilde H_0(\Delta;K)=0$, $\Delta$ is connected.
Every $1$-dimensional connected simplicial complex is pure.
Suppose $\dim \Delta>1$.
By the definition of Serre's conditions,
it is clear that $\lk_\Delta(v)$ satisfies ($S_r$) for every vertex $v$ of $\Delta$.
Thus $\lk_\Delta(v)$ is pure by the induction hypothesis.
On the other hand,
since $\Delta$ is connected,
if $\Delta$ is not pure then
there exists a vertex $v$ of $\Delta$
such that $\lk_\Delta(v)$ is not pure.
Thus $\Delta$ must be pure.
\end{proof}

Lemma \ref{pure} is more generally true.
It is known that if $R$ is a Noetherian catenary local ring satisfying ($S_2$) then $R$ is equidimensional (see \cite[Remark 2.4.1]{Ha}).

Before the proof of Theorem \ref{1}, we recall basic facts on Stanley--Reisner rings.
Let $\Delta$ be a simplicial complex on $[n]$.
The \textit{Stanley--Reisner ideal} $I_\Delta \subset S$ of $\Delta$
is the ideal generated by all squarefree monomials $x_F=\prod_{i \in F} x_i \in S$ with $F \not \in  \Delta$.
The ring $K[\Delta]=S/I_\Delta$ is the \textit{Stanley--Reisner ring} of $\Delta$.
It is known that the Krull dimension of $K[\Delta]$ is equal to $\dim \Delta +1$
and the $h$-vector of $K[\Delta]$ coincides with the $h$-vector of $\Delta$ (cf.\ \cite{St2}).
By virtue of Theorem \ref{2},
the next lemma proves Theorem \ref{1}.

\begin{lemma}
\label{sqmodule}
Let $\Delta$ be a $(d-1)$-dimensional simplicial complex on $[n]$ and $r \geq 2$ an integer.
If $\Delta$ satisfies $(S_r)$ then
$\reg (\Ext_S^{n-i}(K[\Delta],\omega_S)) \leq i-r$ for $i=0,1,\dots,d-1$.
\end{lemma}

\begin{proof}
Fix an integer $1 \leq i \leq d-1$.
Let $N^i = \Ext_S^{n-i}(K[\Delta],\omega_S)$.
We consider the multigraded structure of $N^i$.
For a $\ZZ^n$-graded $S$-module $M$ and for any monomial $u=x_1^{a_1}\cdots x_n^{a_n}$,
we denote by $M_u$ the graded component of degree $(a_1,\dots,a_n) \in \ZZ^n$ of $M$
and $\beta_{p,u}(M)=\dim_K \Tor_p(M,K)_u$.
By the result of Yanagawa \cite{Ya} and Musta\c{t}\u{a} \cite{Mu},
each $N^i$ is a squarefree module.
Thus it follows from \cite[Corollary 2.4]{Ya}
that the multigraded Betti numbers of $N^i$ are concentrated in squarefree degrees,
that is,
$\beta_{p,q}(N^i)= \sum_{F \subset [n],\ |F|=q} \beta_{p,x_F}(N^i)$ for all $p,q$.
On the other hand, by the Hochster's formula \cite[Theorem 5.3.8]{BH},
$$\dim_K (N^i)_{x_F} = \dim_K \tilde H_{i-1-|F|}\big(\lk_\Delta(F);K\big).$$
Since $\Delta$ is pure, $\dim \lk_\Delta(F)=\dim \Delta -|F|$ for every $F \in \Delta$.
Then Serre's condition ($S_r$) implies that $N^i$ has no squarefree elements of degree $\geq i+1-r$.
Thus 
$$K_p(x_1,\dots,x_n;N^i)_{x_F}=0$$
for all $F \subset [n]$ with $|F|\geq p+i+1-r$.
This fact guarantees
$$\beta_
{p,p+j}(N^i)= \sum_{F \subset [n],\ |F|=p+j} \dim_K H_p(x_1,\dots,x_n;N^i)_{x_F}=0$$
for all $p \geq 0$ and $j \geq i+1-r$.
Hence $\reg\hspace{1pt} N^i \leq i-r$.
\end{proof}

\begin{remark}
Recently, generic initial ideals are of great interest in commutative algebra.
Proposition \ref{prop} also gives some information on generic initial ideals.
Let $\Gin(I)$ be the generic initial ideal (cf.\ \cite[\S 15.9]{Ei}) of a homogeneous ideal $I\subset S$
with respect to the reverse lexicographic order.
It was proved by Conca, Herzog and Hibi \cite{CHH} that,
if $\mathrm{char}(K)=0$ then,
for a generic choice of a sequence $\yy=y_1,\dots,y_n$ of linear forms,
$$
\beta_{i,i+j}\big(S/\Gin(I)\big)=
\sum_{k=0}^{n-i} {n-k-1 \choose i-1}\dim_K A_k(\yy;S/I)_j
$$
for all $ i\geq 1$ and $j \in \ZZ$.
(In positive characteristic, the lefthand side is smaller than or equal to the righthand side.)
Proposition \ref{prop} shows if $\reg(\mathrm{Ext}^{n-i}_S(S/I,\omega_S)) \leq i-r$
for $i = 0,1,\dots,d-1$,
where $d$ is the Krull dimension of $S/I$, then
$$\beta_{i,i+j}\big(S/\Gin(I)\big)=0
\ \ \mbox{ for all } i> n-d \mbox{ and } j \leq r-1.$$ 
In particular, this result deduces information on generic initial ideals from a property of local cohomologies.
It would be interesting to study further relations between generic initial ideals
and local cohomologies.
\end{remark}

\section{Proofs of Theorems \ref{3-5} and \ref{main3}}

In this section, we prove Theorems \ref{3-5} and \ref{main3}.

The next simple nice equations due to Swartz \cite[Proposition 2.3]{Sw}
play a crucial role in the proofs.

\begin{lemma}[Swartz]
\label{3-3}
Let $\Delta$ be a $(d-1)$-dimensional pure simplicial complex on $[n]$.
Then
$$ i h_i(\Delta)+ (d-i+1) h_{i-1}(\Delta) = \sum_{v \in [n]} h_{i-1} \big(\lk_\Delta(v)\big)
\ \ \mbox{ for }i=1,2,\dots,d.$$
\end{lemma}

\begin{proof}[Proof of Theorem \ref{main3}]
If $r=1$ then $\sum_{k=1}^{d} h_k(\Delta)=f_{d-1}(\Delta)-1 \geq 0$.
Suppose $r \geq 2$.
Then $\Delta$ is pure by Lemma \ref{pure}.
We use induction on the dimension of $\Delta$.
If $d \leq r$ then the statement is obvious since $\Delta$ is Cohen--Macaulay.
Suppose $d> r$.
By Lemma \ref{3-3},
\begin{eqnarray*}
d \left( \sum_{k=r}^d h_k(\Delta) \right)
&=& r h_r(\Delta)+
\sum_{i=r+1}^d\big\{ i h_i(\Delta)+ (d-i+1)h_{i-1}(\Delta)\big\}\\
&=& r h_r(\Delta) +
 \sum_{i=r+1}^d \left( \sum_{v \in [n]} h_{i-1}\big(\lk_\Delta(v)\big) \right)\\
&=& r h_r(\Delta)
+ \sum_{v \in [n]} \left(\sum_{i=r}^{d-1}  h_{i}\big(\lk_\Delta(v)\big) \right).
\end{eqnarray*}
Since $\lk_\Delta(v)$ satisfies ($S_r$) for every vertex $v$ of $\Delta$,
by the induction hypothesis, $\sum_{i=r}^{d-1}  h_{i}(\lk_\Delta(v)) \geq 0$. 
Also, since $\Delta$ satisfies ($S_r$),
$h_r(\Delta) \geq 0$ by Theorem \ref{1}.
Thus the above equation guarantees $h_r(\Delta)+ h_{r+1}(\Delta)+ \cdots + h_d(\Delta) \geq 0$ as desired.
\end{proof}

While Theorem \ref{main3} easily follows from Theorem \ref{1} and Lemma \ref{3-3}, it provides a non-trivial lower bound for the number of facets.
Recall that if $\dim \Delta=d-1$ then $f_{d-1}(\Delta) = \sum_{k=0}^d(h_k(\Delta))$.
Then Theorem \ref{main3} shows

\begin{corollary}
If a simplicial complex $\Delta$ satisfies Serre's condition $(S_r)$ then
$\Delta$ has at least $\sum_{k=0}^{r-1}h_{k}(\Delta)$ facets.
\end{corollary}

Next, we prove Theorem \ref{3-5}.
We need the following fact on graded Betti numbers.

\begin{lemma}
\label{3-1}
Let $y \in S_1$ be a linear form and $S'=S/yS$.
Let $M$ be a finitely generated graded $S$-module.
If $(0:_M y)_j=0$ for $j \leq r$ then
$$\dim_K \Tor_i^S (M,K)_{i+j} \leq \dim_K \Tor_i^{S'}(M/yM,K)_{i+j}
\ \ \mbox{ for all }i \geq 0 \mbox{ and }j \leq r.$$
\end{lemma}

\begin{proof}
Choose a sequence $\yy'=y_1,\dots,y_{n-1}$ of linear forms such that $y,y_1,\dots,y_{n-1}$ forms a $K$-basis of $S_1$.
Consider the double complex $K_\bullet(\yy';S) \bigotimes_S K_\bullet(y;M)$.
Since its total complex is isomorphic to $K_\bullet (y_1,\dots,y_{n-1},y;M)$,
there is a spectral sequence
$$^2E_{s,t}= H_s \big(\yy';H_t(y;M) \big) \Rightarrow H_{s+t}(y_1,\dots,y_{n-1},y;M) \cong \Tor_{s+t}^S(M,K).$$
Fix an integer $j \leq r$.
By the assumption, $H_i(\yy';(0:_M y))_{i+j} =0$ for all $i \geq 0$.
Then the above spectral sequence implies
\begin{eqnarray*}
\dim_K \Tor_i^S(M,K)_{i+j} 
\!\! &\leq &\!\!
\dim_K \left( H_i\big(\yy';H_0(y;M) \big)_{i+j} \bigoplus H_{i-1}\big(\yy';H_1(y;M)\big)_{i+j} \right)\\
\!\! &=&\!\!
\dim_K \left( H_i(\yy';M/yM )_{i+j} \bigoplus H_{i-1}\big(\yy'; (0:_M y)(-1)\big)_{i+j} \right)\\
\!\! &=&\!\!
\dim_K \Tor_i^{S'}(M/yM,K)_{i+j}
+ \dim_K H_{i-1}\big(\yy'; (0:_M y) \big)_{i-1+j}\\
\!\! &=&\!\!  
\dim_K \Tor_i^{S'}(M/yM,K)_{i+j}
\end{eqnarray*}
as desired. 
\end{proof}

\begin{corollary}
\label{3-2}
Let $M$ be a finitely generated graded $S$-module of Krull dimension $d$
and $\Theta=\theta_1,\dots,\theta_d$ a linear system of parameters of $M$.
If $A_k(\Theta;M)_j=0$ for all $k=0,1,\dots,d-1$ and $j \leq r$ then
$\beta_{i,i+j}(M)=0$ for all $i>n-d$ and $j \leq r$.
\end{corollary}

\begin{proof}
Since $S/\Theta S \cong K[x_1,\dots,x_{n-d}]$ as graded rings,
$\Tor_i^{S/\Theta S}(M/\Theta M,K) =0 $ for $i >n-d$.
By Lemma \ref{3-1},
$\dim_K \Tor_i^S (M,K)_{i+j} \leq \dim_K \Tor_i^{S/\Theta S}(M/\Theta M,K)_{i+j}=0$ for all $i > n-d$ and $j \leq r$, as desired.
\end{proof}

\begin{proof}[Prooof of Theorem \ref{3-5}]
Since $\Delta$ satisfies ($S_t$),
it is enough to consider the case when $r=t$.
We will prove that if $h_{r}(\Delta)=0$ and $\Delta$ satisfies ($S_r$) then $h_{r+1}(\Delta)=0$ and $\Delta$ satisfies ($S_{r+1}$).
If $h_1(\Delta)=0$ then $\Delta$ is generated by a simplex.
Suppose $r \geq 2$.
We use induction on the dimension of $\Delta$.
If $\dim \Delta \leq r-1$ then there is nothing to prove.
Suppose  $\dim \Delta \geq r$.

Since $\lk_\Delta(v)$ satisfies ($S_r$),
by Theorem \ref{1} and Lemma \ref{3-3},
\begin{eqnarray}
\label{3.1}
(r+1) h_{r+1}(\Delta) = \sum_{v \in [n]} h_r \big(\lk_\Delta(v) \big) \geq 0.
\end{eqnarray}
Let $R=K[\Delta]$.
By Proposition \ref{prop} and Lemma \ref{sqmodule}, there exists a linear system of parameters $\Theta=\theta_1,\dots,\theta_d$ of $R$, where $d$ is the Krull dimension of $R$, such that
$A_k(\Theta;R)_j =0$ for $k=0,1,\dots,d-1$ and $j \leq r-1$.
Then Lemma \ref{3}(i) shows $\dim_K (R/\Theta R)_r=h_r(\Delta)=0$.
This fact implies $\dim_K (R/\Theta R)_{r+1}=0$.
Then, by Lemma \ref{3}(ii),
\begin{eqnarray}
\label{3.2}
h_{r+1}(\Delta) = -\sum_{k=0}^{d-1} \dim_K A_k(\Theta;R)_{r} \leq 0.
\end{eqnarray}
By (\ref{3.1}) and (\ref{3.2}), we have
\begin{eqnarray}
\label{3.3}
h_{r+1}(\Delta) = h_r \big(\lk_\Delta(v) \big) = \dim_K A_k(\Theta;R)_{r} = 0
\end{eqnarray}
for all $v \in [n]$ and $k=0,1,\dots,d-1$.

Next, we prove that $\Delta$ satisfies ($S_{r+1}$).
Since $\lk_\Delta(v)$ satisfies ($S_r$) and $h_{r}(\lk_\Delta(v))=0$,
by the induction hypothesis, $\lk_\Delta(v)$ satisfies ($S_{r+1}$).
Thus, for every non-empty face $F \in \Delta$,
$\tilde H_k(\lk_\Delta(F);K)=0$ for $k < \min\{r, \dim \lk_\Delta(F) \}$.
It remains to prove $\tilde H_{r-1}(\Delta;K)=0$.
By (\ref{3.3}),
$A_k(\Theta;M)_j=0$ for all $k=0,1,\dots,d-1$ and $j \leq r$.
Then Corollary \ref{3-2} implies
$$\beta_{i,i+r}\big(K[\Delta]\big)=0 \ \ \mbox{ for } i > n-d.$$
Then, it follows from the Hochster's formula for graded Betti numbers \cite[Theorem 5.5.1]{BH} that
$\tilde H_{r-1}(\Delta;K)=\beta_{n-r,n}(K[\Delta])=0$.
Hence $\Delta$ satisfies ($S_{r+1}$).
\end{proof}

Finally we give a few examples.

\begin{example}
The simplicial complex generated by $\{ \{1,2,3\}, \{3,4,5\}\}$ is pure and connected,
however, does not satisfy ($S_2$) since its link with respect to the vertex $\{3\}$ is not connected.
The simplicial complex $\Delta$ generated by
\begin{eqnarray*}
\left\{
\begin{array}{ll}
\{1,2,3,5\},\{1,2,4,5\},\{1,2,4,6\},\{1,3,4,5\},\{1,3,4,6\}\\
\{1,3,5,6\},\{2,3,4,5\},\{2,3,5,6\},\{2,4,5,6\}\\
\end{array}
\right\}
\end{eqnarray*}
satisfies $(S_2)$ but does not satisfy $(S_3)$ since $\tilde H_1(\lk_\Delta(\{6\});K) \ne 0$.
The simplicial complex $\Gamma$ generated by
\begin{eqnarray*}
\left\{
\begin{array}{ll}
\{1,2,3,5\},\{1,2,4,5\},\{1,2,4,6\},\{1,3,4,5\},\{1,3,4,6\}\\
\{1,3,5,6\},\{2,3,4,6\},\{2,3,5,6\},\{2,4,5,6\}
\end{array}
\right\}
\end{eqnarray*}
satisfies $(S_3)$ but is not Cohen--Macaulay since $\tilde H_2(\Gamma;K) \ne 0$.
\end{example}

\begin{example}
A triangulation $\Delta$ of a compact topological manifold
satisfies Serre's condition ($S_r$) if and only if
$\tilde H_{i}(\Delta;K)=0$ for all $i < r-1$.
For example, a triangulation $\Gamma$ of a $2$-dimensional torus satisfies ($S_2$).
Its suspension 
$\Sigma (\Gamma)=\{ G \cup F: G \subset \{u,v\},\ G \ne \{u,v\},\  F \in \Gamma\}$, where $u$ and $v$ are vertices which are not in $\Gamma$,
is no longer a manifold, but still satisfies ($S_2$).
\end{example}

\begin{example}
A ($d-1$)-dimensional \textit{pseudomanifold} is a $(d-1)$-dimensional pure simplicial complex $\Delta$
satisfying that
\begin{itemize}
\item every ($d-2$)-dimensional face of $\Delta$ is contained in exactly two
facets of $\Delta$;
\item $\Delta$ is \textit{strongly connected}, that is,
for every pair of facets $F$ and $F'$ of $\Delta$,
there exists a sequence of facets $F=F_0,F_1,\dots,F_t=F'$ such that
$|F_k \cap F_{k+1}| = d-1$ for $k=0,1,\dots,t-1$.
\end{itemize}
Pseudomanifolds which satisfy Serre's condition ($S_2$) are known as \textit{normal} pseudomanifolds (see \cite[Section 10]{Ka}).
\end{example}

\section{Monomial ideals and bounds of graded Betti numbers}

The study of graded Betti numbers of monomial ideals is one of current trends in combinatorial commutative algebra.
In this section, we generalize Theorems \ref{1}, \ref{3-5} and \ref{main3}
to monomial ideals,
and consider some upper bounds of graded Betti numbers for monomial ideals
in terms of Serre's conditions.

We first recall an algebraic definition of Serre's condition ($S_r$).
Let $I \subset S$ be a homogeneous ideal.
We say that $R=S/I$ satisfies \textit{Serre's condition} ($S_r$) if
$$\mathrm{depth}\hspace{1pt} R_P \geq \min\{r, \dim R_P\}$$
for all homogeneous prime ideals $P \supset I$.
The first goal of this section is to prove the next theorem.
\begin{theorem}
\label{4-2}
Let $I \subset S$ be a monomial ideal.
If $S/I$ satisfies $(S_r)$ then
\begin{itemize}
\item[(i)] $h_k(S/I) \geq 0$ for $k =0,1,\dots, r$ and $\sum_{k \geq r} h_k(S/I) \geq 0$.
\item[(ii)] if $h_t(S/I)=0$ for some $ t \leq r$ then
$h_k(S/I)=0$ for all $ k \geq t$ and $S/I$ is Cohen--Macaulay.
\end{itemize}
\end{theorem}

To prove the above theorem,
we need the technique, called \textit{polarization},
which associate with each monomial ideal $I$ another squarefree monomial $I^{\pol}$
having the same graded Betti numbers.

\begin{definition}
\label{pola}
Let $I \subset S$ be a monomial ideal
and $G(I)$ the unique set of minimal monomial generators of $I$.
For each monomial $u=x_1^{a_1} \cdots x_n^{a_n}$,
write $\nu (u)=\max \{a_1,\dots,a_n\}$.
Set $N= \max\{ \nu (u): u \in G(I)\}$.
Let $V_k=\{x_{k,1},\cdots,x_{k,N}\}$, where $k=1,2,\dots,n$,
be sets of variables and let $A=K[V]$ be the polynomial ring with the set of
variables $V=\bigcup_{k=1}^n V_k$.
For any monomial $u =x_1^{a_1} \cdots x_n^{a_n} \in S$
with $\nu(u) \leq N$,
let
$$\pol (u) = \prod_{1 \leq k \leq n,\ a_k \ne 0} (x_{k,1} \cdots x_{k,a_k}) \in A.$$
The squarefree monomial ideal $I^{\pol} \subset A$ generated by
$\{ \pol(u) : u \in G(I)\}$ is called the \textit{polarization} of $I$.
It is known that $I$ and $I^{\pol}$ have the same graded Betti numbers,
say, $\beta_{i,j}^S(S/I)= \beta_{i,j}^{A}(A/I^{\pol})$ for all $i,j$ (cf.\ \cite[Lemma 4.2.16]{BH}),
where $\beta_{i,j}^{A}(A/I^{\pol})=\dim_K \Tor_i^A(A/I^{\pol},K)_j$ are the graded Betti numbers of $A/I^{\pol}$ over $A$.
In particular, by the Auslander-Buchsbaum formula \cite[Theorem 1.3.3]{BH}
\begin{eqnarray*}
\label{4.2}
\mathrm{depth}(A/I^{\pol})= \mathrm{depth}(S/I) +n(N-1).
\end{eqnarray*}
\end{definition}

We also need the following easy fact.

\begin{lemma}
\label{4-1}
Let $\Delta$ be a simplicial complex on $[n]$.
For every vertex $v$ of $\Delta$ one has
$\mathrm{depth}(K[\lk_\Delta(v)]) \geq \mathrm{depth}(K[\Delta]) -1$. 
\end{lemma}

\begin{proof}
The Hochster's formula \cite[Theorem 5.3.8]{BH} says
\begin{eqnarray}
\label{4.1}
\begin{array}{lll}
\mathrm{depth}\big( K[\Delta]\big)&=&
\min\big\{i: H^i_\mideal \big(K[\Delta]\big) \ne 0\big\} \medskip\\
&=&
\min\big\{i: \tilde H_{i-1-|F|}\big(\lk_\Delta (F);K\big) \ne 0 \mbox{ for some }F \in \Delta\big\}.
\end{array}
\end{eqnarray}
The above equation immediately implies the desired inequality.
\end{proof}

\begin{proof}[Proof of Theorem \ref{4-2}]
Without loss of generality,
we assume that $I$ has no elements of degree $1$.
Let $I^{\pol} \subset A$ be the polarization of $I$ defined in Definition \ref{pola}.
Since $I^{\pol}$ is a squarefree monomial ideal, there exists a simplicial complex $\Delta$ whose Stanley--Reisner ideal is equal to $I^{\pol}$.
Since $S/I$ and $A/I^{\pol}$ have the same graded Betti numbers,
they have the same $h$-vector and $S/I$ is Cohen--Macaulay if and only if $A/I^{\pol}$ is Cohen--Macaulay.
Thus it is enough to prove that $\Delta$ satisfies ($S_r$).

We regard $\Delta$ as a simplicial complex with the vertex set
$V=\bigcup_{k=1}^n V_k$, where $V_k=\{x_{k,1},\dots,x_{k,N}\}$.
%
%
Let $F \subset V$ be a face of $\Delta$.
To prove Theorem \ref{4-2}, by (\ref{4.1}), it is enough to prove
$$\mathrm{depth} \big(K[\lk_\Delta(F)]\big) \geq \min\big\{r, \dim \lk_\Delta (F)+1\big\}.$$
Write $F=F_1 \cup F_2 \cup \cdots \cup F_n$ such that each $F_k \subset V_k$.
We may assume that $F$ is a subset of the form
$$F=F_1 \cup \cdots \cup F_s \cup V_{s+1} \cup \cdots \cup V_n$$
where $1 \leq s \leq n$ and $F_k$ is a proper subset of $V_k$ for $k =1,\dots,s$.

Let $S'=K[x_1,\dots,x_s]$
and $I'=(I:x_{s+1}^N \cdots x_n^N) \cap S'$.
Thus $I'$ is the monomial ideal of $S'$ whose generators are obtained from $G(I)$
by substituting $x_k$ by $1$ for $k=s+1,\dots,n$.
Let $P=(x_1,\dots,x_s)$.
Then
$$ \mathrm{depth}(S'/I')= \mathrm{depth}(S/I)_P
\geq \min\big\{r,\dim (S/I)_P\big\}
= \min\big\{r,\dim (S'/I')\big\}.$$
Let $J \subset B=K[\bigcup_{k=1}^s V_k]$ be the monomial ideal generated
by $\{ \pol(u):u \in G(I')\}$.
Consider the set $\Gamma$ of all squarefree monomials in $B$ which are not
in $J$.
We regard $\Gamma$ as a simplicial complex by identifying squarefree monomials with sets of variables.
Since $J$ is a squarefree monomial ideal,
$J$ is generated by the monomials in $G(I_\Gamma)$ and the set of variables of $B$ which are not in $\Gamma$.
Thus $\mathrm{depth}(K[\Gamma])=\mathrm{depth}(B/J)$.
Since $J$ is the polarization of $I'$,
$$\mathrm{depth}\big(K[\Gamma]\big)=\mathrm{depth}(B/J)=
\mathrm{depth}(S'/I') + s(N-1).$$
Also, by the constructions of $\Delta$ and $\Gamma$,
a routine computation implies 
$$\lk_\Delta(V_{s+1}\cup \cdots \cup V_n)=\Gamma.$$
Hence $\lk_\Delta(F)=\lk_\Gamma(F_1\cup \cdots \cup F_s)$.
Then what we must prove is
$$\mathrm{depth}\big(K[\lk_\Gamma(F_1\cup \cdots \cup F_s)]\big) \geq \min\big\{r,\dim \lk_\Gamma(F_1\cup \cdots \cup F_s)+1 \big\}.$$

Suppose $\mathrm{depth}(S'/I') = \dim (S'/I')$.
Then $S'/I'$ is Cohen--Macaulay.
Since polarization does not change the Cohen--Macaulay property,
$B/J \cong K[\Gamma]$ is Cohen--Macaulay.
Then, since the Cohen--Macaulay property is preserved under taking links,
$\lk_\Gamma(F_1\cup \cdots \cup F_s)$ is Cohen--Macaulay.
Hence $\mathrm{depth}(K[\lk_\Gamma(F_1\cup \cdots \cup F_s)])=\dim (\lk_\Gamma(F_1\cup \cdots \cup F_s))+1$.

Suppose $\mathrm{depth}(S'/I') \geq r$.
Then $\mathrm{depth}(K[\Gamma]) \geq r +s(N-1)$.
Since $F_k$ is a proper subset of $V_k$ for $k=1,2,\dots,s$,
$|F_1 \cup \cdots \cup F_s| \leq s(N-1)$.
Then Lemma \ref{4-1} guarantees
$ \mathrm{depth}(K[\lk_\Gamma(F_1\cup \cdots \cup F_s)]) \geq r$.
\end{proof}

Next, we consider upper bounds of graded Betti numbers.
On upper bounds of graded Betti numbers of homogeneous ideals,
one of influential results is the Bigatti-Hulett-Pardue theorem \cite{Bi,Hu,Pa} which proved that
a lex ideal $L$ in $S$ has the largest graded Betti numbers among all homogeneous ideals in $S$
having the same Hilbert series as $L$.
Recall that a monomial ideal $I \subset S$ is said to be \textit{lex}
if, for all monomials $u \in I$ and $v >_\lex u$ with $\deg v =\deg u$, one has $v \in I$,
where $>_\lex$ is the lexicographic order induced by the ordering $x_1> \cdots >x_n$.
It is known that, for every homogeneous ideal $I \subset S$,
there exists the unique lex ideal $I^\lex \subset S$ such that
$S/I$ and $S/I^\lex$ have the same Hilbert series.
The Bigatti-Hulett-Pardue theorem says $\beta_{i,j}(S/I)\leq \beta_{i,j}(S/I^\lex)$ for all $i,j$.

For a homogeneous ideal $I \subset S$ such that $S/I$ is Cohen--Macaulay,
it is easy to see that the Bigatti-Hulett-Pardue theorem can be improved as follows:
Let $S'=K[x_1,\dots,x_{n-d}]$,
where $d$ is the Krull dimension of $S/I$,
and $L \subset S'$ the lex ideal satisfying that $\dim_K(S'/L)_k=h_k(S/I)$ for all $k$.
Then $\beta_{i,j}^S(S/I)\leq \beta_{i,j}^{S'}(S'/L)$ for all $i$ and $j$.
(Note that $\beta_{i,j}^{S'}(S'/L)$ are much smaller than $\beta_{i,j}^{S}(S/I^\lex)$.)

The next theorem generalize the above fact for monomial ideals $I \subset S$ such that $S/I$
satisfies ($S_r$).

\begin{theorem}
\label{lex}
Let $I \subset S$ be a monomial ideal and $d$ the Krull dimension of $R=S/I$.
Suppose that $R$ satisfies $(S_r)$.
Then there exists a lex ideal $L \subset S'=K[x_1,\dots,x_{n-d}]$ such that
$\dim_K (S'/L)_k= h_k(R)$ for $k =0,1,\dots,r$ and
$$\beta_{i,i+j}^S(S/I) \leq \beta_{i,i+j}^{S'}(S'/L)
\ \ \mbox{for all $i \geq 0$ and $j \leq r-1$}. $$
\end{theorem}

\begin{proof}
The statement is obvious when $r=1$.
Suppose $r \geq 2$.
Since polarization does not change graded Betti numbers,
in the same way as in the proof of Theorem \ref{4-2},
it is enough to prove the statement for Stanley--Reisner rings.

Let $\Delta$ be a ($d-1$)-dimensional simplicial complex on $[n]$ which satisfies ($S_r$) and $R=K[\Delta]$.
By Proposition \ref{prop} and Lemmas \ref{sqmodule} and \ref{3-1}, there exists a linear system of parameters $\Theta$
of $R$ such that $\beta_{i,i+j}^S(R) \leq \beta_{i,i+j}^{S/\Theta S}(R/\Theta R)$
for all $i \geq 0$ and $j \leq r-1$.
Since $S/\Theta S \cong S'=K[x_1,\dots,x_{n-d}]$,
there exists a homogeneous ideal $J \subset S'$ such that
$\beta_{i,j}^{S/\Theta S} (R/\Theta R)= \beta_{i,j}^{S'}(S'/J)$
for all $i,j$.
Since $R/\Theta R$ and $S'/J$ have the same Hilbert series,
it follows from Lemma \ref{3} that $\dim_K (S'/J)_k = \dim_K (R/\Theta R)_k=h_k(R)$ for $k \leq r$.
Then the Bigatti-Hulett-Pardue theorem guarantees that the lex ideal $L \subset S'$ having the same Hilbert series as $J$ satisfies the desired conditions.
\end{proof}

Finally, we give a simpler upper bound of graded Betti numbers which only depends on codimension, and determine when the equality holds.
A homogeneous ideal $I \subset S$ is said to have a \textit{$k$-linear resolution}
if $I$ is generated by elements of degree $k$ and $\reg(I)=k$.

\begin{theorem}
\label{ub}
Let $I \subset S$ be a monomial ideal and $d$ the Krull dimension of $R=S/I$.
Suppose that $R$ satisfies $(S_r)$.
Let $S'=K[x_1,\dots,x_{n-d}]$ and $\mathfrak n =(x_1,\dots,x_{n-d}) \subset S'$. Let $k \leq r$ be an integer. Then
\begin{itemize}
\item[(i)]
$ \beta_{i,i+k-1}(S/I) \le \beta_{i,i+k-1}^{S'}(S'/\mathfrak n^{k})$ for all $i \geq 0$.
\item[(ii)]
the following conditions are equivalent.
\begin{itemize}
\item[(a)] $\beta_{i,i+k-1}(S/I) = \beta_{i,i+k-1}^{S'}(S'/\mathfrak n^{k})$
for some $1 \leq i \leq n-d$;
\item[(b)] $\beta_{i,i+k-1}(S/I) = \beta_{i,i+k-1}^{S'}(S'/\mathfrak n^{k})$
for all $i \geq 0$;
\item[(c)] $S/I$ is Cohen--Macaulay and $I$ has a $k$-linear resolution.
\end{itemize}
\end{itemize}
\end{theorem}

\begin{proof}
By Theorem \ref{lex},
there exists a lex ideal $L$ of $S'$ satisfying that
$\dim_K (S'/L)_j=h_j(S/I)$ and
$\beta_{i,i+j-1}^S(S/I) \leq \beta_{i,i+j-1}^{S'}(S'/L)$ for all $i \geq 0$ and $j \leq r$.
On the other hand, by the Eliahou-Kervaire formula \cite{EK},
\begin{eqnarray}
\label{ek}
&&\beta_{i,i+k-1}^{S'}(S'/L) = \sum_{{u \in G(L)}\atop{\deg u=k}} { m(u) -1 \choose i-1} \leq \sum_{{u \in G(\mathfrak n^k)}}{ m(u) -1 \choose i-1} =
\beta_{i,i+k-1}^{S'}(S'/\mathfrak n^{k})
\end{eqnarray}
for all $i \geq 1$,
where $m(u)$ is the maximal integer $\ell$ such that $x_\ell$ divides $u$.
The above inequality proves (i).

Next, we show (ii).
The implication $(c) \Rightarrow (b)$ is well-known and $(b) \Rightarrow (a)$ is obvious.
We will prove $(a) \Rightarrow (c)$.

Suppose $\beta_{i,i+k-1}^{S'}(S'/\mathfrak n^{k})=\beta_{i,i+k-1}^S(S/I)$
for some integer $1 \leq i \leq n-d$.
By (\ref{ek}),
$$
\beta_{i,i+k-1}^{S'}(S'/\mathfrak n^{k})=\beta_{i,i+k-1}^S(S/I) \leq \beta_{i,i+k-1}^{S'}(S'/L) \leq \beta_{i,i+k-1}^{S'}(S'/\mathfrak n^{k}).$$
Thus $\beta_{i,i+k-1}^{S'}(S'/L) = \beta_{i,i+k-1}^{S'}(S'/\mathfrak n^{k}).$

We claim $L= \mathfrak n^k$.
By (\ref{ek}),
$\beta_{i,i+k-1}^{S'}(S'/L) = \beta_{i,i+k-1}^{S'}(S'/\mathfrak n^{k})$
implies that $G(L)$ contains all monomials $u \in S'$ of degree $k$
with $m(u) \geq i$.
Then $x_{n-d}^k \in L$.
Since $L$ is lex, $L \supset \mathfrak n^k$.
It remains to prove that $L$ has no generators of degree $<k$.
If $L$ has a generator $v \in G(L)$ of degree $k-s$ with $s>0$ then
$v x_{n-d}^s \not \in G(L)$.
This contradicts the fact that $G(L)$ contains all monomials $u \in S'$ of degree $k$
with $m(u) \geq i$.

Since $L=\mathfrak n^k$, we have $h_k(S/I)=\dim_K (S'/L)_k=0$.
Then, by Theorem \ref{4-2}(ii), $S/I$ is Cohen--Macaulay
and $h_j(S/I)=0$ for all $j \geq k$.
In particular, $\reg (I) \leq k$ since $\reg (J)= \max\{j: h_j(S/J) \ne 0\} +1$ for any homogeneous ideal $J \subset S$ such that $S/J$ is Cohen--Macaulay.
Also, since $ \beta_{i,i+j}^S(S/I) \leq \beta_{i,i+j}^{S'}(S'/L)=\beta_{i,i+j}^{S'}(S'/\mathfrak n^k)$ for all $j \leq k$,
$I$ has no generators of degree $<k$.
Hence $I$ has a $k$-linear resolution.
\end{proof}

\begin{example}
The bound considered in Theorem \ref{ub} is false
under the weaker assumption that $\beta_{i,i+k-1}(S/I)=0$ for all $i > n-d$ and $k \leq r$.
Let $\Delta$ be a simplicial complex generated by
$$\big\{ \{1,2,3\},\{3,4,5\},\{5,6,7\},\{7,8,1\}\big\}.$$
Then $\beta_{i,i+k}(K[\Delta])=0$ for all $i > \mathrm{codim}(K[\Delta])=5$ and $k \leq 1$,
however, $$\beta_{1,2}\big(K[\Delta]\big)=16>\beta_{1,2}\big(K[x_1,\dots,x_5]/(x_1,\dots,x_5)^2\big)=15.$$
\end{example}

We are not sure whether Theorem \ref{4-2} holds for all homogeneous ideals.
It would be interesting to find classes of homogeneous ideals for which
Theorem \ref{4-2} holds.
\bigskip

\section*{Acknowledgements}
This work is supported by KAKENHI 20540047.
The first author is supported by JSPS Research Fellowships for Young Scientists.
We are grateful to Ryota Okazaki for helpful discussions
about spectral sequences and Matlis duality.


\begin{thebibliography}{1}
\bibitem[AH]{AH}
A. Aramova and J. Herzog,
Almost regular sequences and Betti numbers,
\textit{Amer. J. Math.} \textbf{122} (2000), 689--719. 

\bibitem[Bi]{Bi}
A. Bigatti, 
Upper bounds for the Betti numbers of a given Hilbert function, 
\textit{Comm. Algebra} \textbf{21} (1993), 2317--2334. 


\bibitem[BH]{BH}
W. Bruns and J. Herzog,
Cohen--Macaulay rings, Revised Edition, Cambridge University Press, Cambridge, 1998.

\bibitem[CHH]{CHH}
A. Conca, J. Herzog and T. Hibi,
Rigid resolutions and big Betti numbers,
\textit{Comment. Math. Helv.} \textbf{79} (2004), 826--839. 


\bibitem[Ei]{Ei}
D. Eisenbud,
Commutative algebra with a view toward algebraic geometry,
Grad. Texts in Math., vol. 150, Springer-Verlag, New York, 1995.

\bibitem[EK]{EK}
S. Eliahou and M. Kervaire,
Minimal resolutions of some monomial ideals,
\textit{J. Algebra} \textbf{129} (1990), 1--25.

\bibitem[Ha]{Ha}
R. Hartshorne,
Complete intersections and connectedness,
\textit{Amer. J. Math.} \textbf{84} (1962), 497--508. 

\bibitem[Hu]{Hu}
H. Hulett, Maximum Betti numbers of homogeneous ideals with 
a given Hilbert function, 
\textit{Comm. Algebra} \textbf{21} (1993), 2335--2350. 


\bibitem[Ka]{Ka}
G. Kalai,
Rigidity and the lower bound theorem I,
\textit{Invent. Math.} \textbf{88} (1987), 125--151. 



\bibitem[Mu]{Mu}
M. Musta\c{t}\u{a},
Local cohomology at monomial ideals,
\textit{J. Symbolic Comput.} \textbf{29} (2000), 709--720. 


\bibitem[Pa]{Pa}
K. Pardue,
Deformation classes of graded modules and maximal Betti numbers,
\textit{Illinois J. Math. }\textbf{40} (1996), 564--585.


\bibitem[Sc1]{Sc}
P. Schenzel,
Applications of Koszul homology to numbers of generators and syzygies,
\textit{J. Pure Appl. Algebra} \textbf{114} (1997), 287--303. 

\bibitem[Sc2]{S2}
P. Schenzel,
On the use of local cohomology in algebra and geometry,
in: J. Elias, J.M. Giral, R.M. Mir\'o-Roig and S. Zarzuela (Eds.),
Six Lectures on Commutative Algebra,
Progr. Math., vol. 166, Birkh\"auser, Basel, 1998,
pp. 241--292.


\bibitem[St1]{St1} R.P. Stanley,
Hilbert functions of graded algebras,
\textit{Advances in Math.} \textbf{28} (1978), 57--83.


\bibitem[St2]{St2} R.P. Stanley,
Combinatorics and Commutative Algebra, Second Edition,
Birkh\"auser, Boston, 1996.

\bibitem[Sw]{Sw}
E. Swartz,
Lower bounds for $h$-vectors of $k$-CM, independence, and broken circuit complexes,
\textit{SIAM J. Discrete Math.}  \textbf{18}  (2004/05), 647--661.


\bibitem[TY]{TY}
N. Terai and K. Yoshida,
Buchsbaum Stanley--Reisner rings with minimal multiplicity,
\textit{Proc. Amer. Math. Soc.} \textbf{134} (2006), 55--65.

\bibitem[Ya]{Ya}
K. Yanagawa,
Alexander duality for Stanley--Reisner rings and squarefree $\mathbb{ N}\sp n$-graded modules,
\textit{J. Algebra} \textbf{225} (2000), 630--645. 

\end{thebibliography}
\end{document}